\documentclass[11pt]{amsart}

\usepackage{style}
\DeclareMathOperator{\Diag}{Diag}

\DeclareMathOperator{\sech}{sech}
\DeclareMathOperator{\K}{\mathbb{K}}
\DeclareMathOperator{\SU}{\mathrm{SU}}
\DeclareMathOperator{\Herm}{Herm}

\newcommand{\KP}{\mathbb{KP}}
\newcommand{\KPd}{\mathbb{KP}^{d-1}}
\newcommand{\KPdual}{(\mathbb{KP}^{d-1})^\ast}
\newcommand{\Abs}[1]{\ensuremath{\left\lVert #1 \right\rVert}}

\title{Certifying Anosov representations}
\author{J. Maxwell Riestenberg}
\date{\today}

\begin{document}

\begin{abstract}
    By providing new finite criteria which certify that a finitely generated subgroup of $\SL(d,\R)$ or $\SL(d,\C)$ is projective Anosov, we obtain a practical algorithm to verify the Anosov condition. 
    We demonstrate on a surface group of genus 2 in $\mathrm{SL}(3,\mathbb{R})$ by verifying the criteria for all words of length 8.
    The previous version required checking all words of length $2$ million.
\end{abstract}

\keywords{discrete subgroups of Lie groups,
Anosov subgroups,
symmetric spaces,
coarse geometry,
hyperbolic groups}

\maketitle

\section{Introduction}

In general, it is a hard problem to determine whether a finite subset of $\SL(d,\R)$ generates a discrete subgroup.
Aside from the special case of $\SL(2,\R)$ \cite{GM1991algorithm}, the discreteness property cannot be decided by an algorithm \cite{kapovich2016discreteness} in the Blum-Schub-Smale model of computation.
Nonetheless, it remains possible to verify stronger properties than discreteness via certain geometric algorithms \cite{Kap23}.
This paper concerns a numerically stable algorithm that can verify the \emph{Anosov} property, which is stronger than discreteness, in finite time. 

\emph{Anosov subgroups}, introduced by Labourie \cite{Lab06} and Guichard-Wienhard \cite{GW12}, provide a rich yet tractable source of examples of infinite covolume discrete subgroups of higher rank Lie groups.
Their importance is underscored by their central role in \emph{higher Teichm\"uller theory} \cite{Hit92, fockModuliSpacesLocal2006a, GW18} as well as the rich examples of dynamical systems \cite{canary2024measurestransverse,ELO23anosovmixing,Sam24ergodic,delarue2025locally} and geometric structures on manifolds \cite{KLP18a, dancigergueritaudkassel24convexcocompactrealprojective, Kas18} that they provide.
Anosov subgroups generalize convex cocompact subgroups of isometries of hyperbolic space to higher rank.
In particular, they can be characterized in terms of the coarse geometry of their action on the associated symmetric space \cite{KLP17,KLP18b,BPS19}.

\medskip

In the present article, we present the first practical algorithm which certifies that a finitely generated linear group is projective Anosov. 
The core of the algorithm, and main result of this paper, is Theorem \ref{thm:straight and spaced implies undistorted}, which establishes new finite criteria for a sequence in the symmetric space to have coarsely linear singular value gaps.
We demonstrate the practicality of the algorithm by verifying the criteria on an example of a surface group in $\SL(3,\R)$ by checking words of length $8$. 

\medskip

The approach is based on a coarse geometric characterization of Anosov subgroups due to Kapovich-Leeb-Porti \cite{KLP14,KapovichLeebPorti2023localtoglobal}. 
They proved a local-to-global property for \emph{Morse quasigeodesics}, and described an algorithm to certify the Anosov property of a finitely generated subgroup of a semisimple Lie group. 
If a subgroup is Anosov, their algorithm will stop and certify so in finite time; otherwise, the subgroup is not Anosov and the algorithm will run forever.
The author made their algorithm effective \cite{riestenberg2024quantified} by supplying their arguments with explicit estimates. 
While effective, that version of the algorithm was impractical, requiring the user to verify a condition on words of length $2$ million even on a simple example of an Anosov surface group in $\SL(3,\R)$. 

The bulk of the improvement is due to Lemma \ref{lem:angle-distance formula}, which provides a formula relating $\zeta$-angles (Definition \ref{def:zeta angles}) and distances to parallel sets (\S \ref{sec:parallel sets}) of the relevant type. 
Such a formula cannot hold for parallel sets in general (Remark \ref{rem:no formula}).
This is why the present paper only directly deals with projective Anosov subgroups of $\SL(d,\R)$ or $\SL(d,\C)$, rather than general $\Theta$-Anosov subgroups of an arbitrary semisimple Lie group $G$.
Fortunately, verifying the $\Theta$-Anosov property of such a subgroup reduces to the projective Anosov case in a straightforward way \cite[Section 4]{GW12}, so the approach here may be applied in the general case.
We mention that a different algorithm to verify the $\Theta$-Anosov property was recently obtained by the author and Davalo \cite{davaloriestenberg2024dirichletanosov}.
It is based on Dirichlet domains with respect to Finsler metrics, and guaranteed to eventually terminate for $\Theta$-Anosov subgroups in certain cases, e.g.\ $n$-Anosov subgroups of $\Sp(2n,\R)$ or Borel Anosov subgroups of $\SL(d,\R)$ or $\SL(d,\C)$. 

The local-to-global principle  of Kapovich-Leeb-Porti relies on a Theorem which guarantees that sufficiently straight and spaced sequences are Morse quasigeodesics.
Our Theorem \ref{thm:straight and spaced implies undistorted} is a similar but slightly different statement that guarantees a sequence is \emph{$d_\alpha$-undistorted} (Definition \ref{def:definitions for sequences}), which means that the first singular value gap grows coarsely linearly. 
A Morse quasigeodesic is always $d_\alpha$-undistorted but the converse does not hold for arbitrary sequences. 
Accordingly, the criterion we obtain here has weaker assumptions and a weaker conclusion than e.g.\ \cite[Theorem 3.18]{KapovichLeebPorti2023localtoglobal}. 
See the beginning of Section \ref{sec: straight and spaced} for a more precise discussion.
Besides the statement, there are also some technical differences with the proof here compared with their proof and that of \cite{riestenberg2024quantified}. 
For example, Kapovich-Leeb-Porti work with an $\iota$-invariant model simplex (where $\iota$ is the ``opposition involution") and we crucially drop that assumption here.
Compared to the proof in \cite{riestenberg2024quantified} we make use of further auxiliary parameters and incorporate new estimates (Lemma \ref{lemma: dalpha and busemanns} and Lemma \ref{lem:lipschitz constant of zeta}) and the key angle-to-distance formula, Lemma \ref{lem:angle-distance formula}, mentioned above.

\medskip

Applying the local-to-global principle amounts to calculating various geometric quantities in the associated symmetric space.
An implementation by the author is available at \cite{Rie24code}, and a faster implementation by Teddy Weisman is available at \cite{Weis24code}.
Both implementations are in Python. 
KBMAG \cite{kbmag1.5.11} is required to produce an automoton and enumerate all geodesic words of length $8$. 
Strictly speaking, neither computation is rigorous, in the sense that there is no guarantee on the numerical precision in the calculation; however, the results presented here are compatible with the computation in \cite{Rie23verifying} based on hyperbolic geometry, so are expected to be accurate.
An implementation with numerical guarantees is necessary to use these techniques to rigorously prove the Anosov property.
A generalized implementation, especially one which can accept approximate generating sets and provide rigorous guarantees for the numerical precision, is an appealing avenue for future work.

\section{Acknowledgments}
I remain grateful to my PhD advisor, Jeff Danciger, who first suggested to me the problem of making the Kapovich-Leeb-Porti algorithm effective. 
I warmly thank Jacques Audibert and the anonymous referee for helpful remarks on the previous draft which have improved the exposition.

\section{Symmetric space reminders}

For necessary background on symmetric spaces we refer to \cite{helgason79,eberlein96}.

\subsection{The model and metric}

Let $\K = \R$ or $\C$ and let $\X$ denote the symmetric space associated to $G=\SL(d,\K)$. 
Concretely, a model for $\X$ is given by
$$ \X = \{ X \in \Herm(d,\K) : X \gg 0, \det(X) =1 \} $$
where $\Herm(d,\K)$ denotes Hermitian $d\times d$ matrices with entries in $\mathbb{K}$ and $X \gg 0$ means that $X$ is positive definite.
There is a natural action of $\SL(d,\K)$ on $\X$ given by $g.X=gXg^\dagger$, where $g^\dagger$ denotes the conjugate transpose of $g$. 
We denote the stabilizer of the identity matrix $I_d$ by $K=\SU(d,\K)$, which is simply $\SO(d)$ if $\K=\R$ and $\SU(d)$ if $\K=\C$.
This model corresponds to the natural identification of $\X$ with the space of inner products on $(\K^d)^\ast$ modulo scaling. 

There is a unique $\SL(d,\K)$-invariant Riemannian metric on $\X$ up to scale.
With such a metric, $\X$ becomes a Riemannian symmetric space: for each $p \in \X$ there exists a (unique) involutive isometry $S_p \colon \X \to \X$ with $p$ as an isolated fixed point.
Moreover $\X$ is a symmetric space of non-compact type and in particular $\X$ is a Cartan-Hadamard manifold. 

In this paper we use the Riemannian metric: 
$$ \forall X,Y \in T_p\X \subset \Herm(d,\K), \quad \metric{X}{Y}_p = \frac12 \Tr(p^{-1}Xp^{-1}Y) .$$ 
The specific choice of metric will simplify some of the formulas to follow. 
For this metric, when $d\ge 3$, the sectional curvature of $\X$ has image $[-1,0]$, and for $d=2$, $\X$ has constant sectional curvature $-1$.

\medskip

\subsection{Geodesics and vector-valued distance}

If $c\colon \R \to \X$ is a geodesic, then there is a unique $1$-parameter subgroup $t \mapsto \exp(tX)$ in $\SL(d,\K)$ such that $c(t) = \exp(tX).c(0)$.
We now recall the well-known fact that every geodesic in a symmetric space can be put into a standard position.
Let $\mathfrak{a}$ denote the set of real diagonal $d\times d$ matrices of trace $0$, and let $\mathfrak{a}^+$ denote the subset of $\mathfrak{a}$ whose diagonal entries are non-increasing. 

\begin{fact}
    If $c$ is a geodesic with $c(0)=I_d$ then there exist $k,k' \in \SU(d,\K)$ and $A \in \mathfrak{a}^+$ such that $c(t) = k \exp(tA)k'.I_d$.
    The element $A \in \mathfrak{a}^+$ is uniquely determined.
\end{fact}

This leads to a vector-valued invariant for tangent vectors and a vector-valued distance for pairs of points.
Indeed, if $v \in T_p\X$, then let $c$ be the geodesic with $c'(0)=v$, and set $\vec{d}(v)=A \in \mathfrak{a}^+$. 
Similarly, for a pair of points $p,q \in \X$ we let $\vec{d}(p,q)$ denote the unique $A \in \mathfrak{a}^+$ associated to the geodesic segment $pq$. 

\subsection{Scale of the metric}\label{subsec:scale of metric}
The map $\mathfrak{a} \to \mathbb{X}$ given by $A \mapsto \exp(A).I_d$ becomes an isometric embedding when $\mathfrak{a}$ is endowed with the inner product 
$$ A,B \in \mathfrak{a} \mapsto 2 \Tr(AB) ,$$
because the derivative of the orbit map at $I_d$ restricted to symmetric matrices is multiplication by $2$. 
Note that the Killing form of $\mathfrak{sl}(d,\K)$ is given by 
$$ X,Y \in \mathfrak{sl}(d,\K) \mapsto 2d \Tr(XY) ,$$
so the metric we use in this paper is $\frac1{d}$ times the Riemannian metric induced by the Killing form. 
As a result, the present paper has some slightly different formulas compared to \cite{riestenberg2024quantified}. 
Specifically, each appearance of $\kappa_0$ there is replaced with $1$ here. 

\medskip

\subsection{The visual boundary}

The \emph{visual boundary} of $\X$, denoted $\partial_{\rm vis} \X$, is the set of unit-speed geodesic rays up to asymptotic equivalence. 
For any $p \in \X$, the exponential map yields a homeomorphism $S(T_p \X) \to \partial_{\rm vis} \X$ with respect to the visual topology on $\partial_{\rm vis} \X$.
Any isometry of $\X$ extends to a homemorphism of $\partial_{\rm vis}\X$. 
A fundamental domain for the action of $\SL(d,\K)$ on $\partial_{\rm vis} \X$ is given by the set of unit vectors in $\mathfrak{a}^+$, which we denote by $\sigma_{mod}$ (called the \emph{model (spherical) Weyl chamber}).
Each $\SL(d,\K)$-orbit in $\partial_{\rm vis}\X$ meets $\sigma_{mod}$ exactly once, so there is an induced map $\partial_{\rm vis} \X \to \sigma_{mod}$.
The image of an ideal point under this map is called its \emph{type}.

\subsection{Projective space}

This paper concerns projective Anosov subgroups, so projective space and its dual will play a distinguished role.
We must first describe how those spaces are embedded naturally in the visual boundary.

Let 
\begin{equation}\label{eqn: Z}
    Z = \frac1{\sqrt{2d(d-1)}}
                \begin{bmatrix}  
                    d-1 &  &  &  &  \\
                    & -1 &  &  &  \\
                    & & -1 & & \\
                    & & & \ddots & \\                 
                    & & & & -1 \\
                \end{bmatrix} \in \mathfrak{a}^+
\end{equation}
and consider the unit-speed geodesic ray $c_Z(t)\coloneqq \exp(tZ).I_d$. 
The orbit $G.[c_Z]\subset \partial_{\rm vis} \X$ is a copy of the projective space $\KPd$ with equivariant diffeomorphism given by $g[c_Z] \mapsto g.[e_1]$, where $e_1 \in \mathbb{K}^d$ spans the $(d-1)$-eigenspace of $Z$. 
The same correspondence is obtained by observing that the Hermitian matrices along a ray $g \circ c_Z$ limit, up to rescaling, to a rank $1$ Hermitian matrix, which may be written $vv^\dagger$ for a unique nonzero $v \in \mathbb{K}^d$ up to scale.
Simililarly $G.[c_{-Z}]\subset \partial_{\rm vis} \X$ is a copy of the dual projective space $\KPdual$.
In this case the correspondence is given by the radical of the limiting rescaled Hermitian matrix.
For the rest of the paper, we will abuse notation by taking this identification to be implicit. 
In particular, we will frequently write $\measuredangle_q(\hat{\tau},\tau)$ for the Riemannian angle at $q \in \X$ between the ideal points $\tau \in \KPd$ and $\hat{\tau} \in \KPdual$. 

For consistency with the notation of \cite{KLP14, riestenberg2024quantified}, we will use $\zeta$ to denote the type of $[c_Z]$ and $\iota \zeta$ to denote the type of $[c_{\iota Z}]$, where
$$ \iota Z = \frac1{\sqrt{2d(d-1)}}
                \begin{bmatrix}  
                    1 &  &  &  &  \\
                    & 1 &  &  &  \\
                    & & 1 & & \\
                    & & & \ddots & \\                 
                    & & & & 1-d \\
                \end{bmatrix} \in \mathfrak{a}^+
.$$

\subsection{Regular directions and root pseudometric}\label{sec:regularity and pseudometric}
A geodesic segment $pq$ (resp.\ tangent vector $v$) is \emph{$\zeta$-regular} if its vector-valued invariant $A = \Diag(a_1,\dots,a_d)$ has $a_1 -a_2 >0$.
This occurs if and only if $qp$ (resp. $-v$) is \emph{$\iota \zeta$-regular}, i.e.\ the vector-valued invariant is $B$ with $b_{d-1}-b_d >0$. 
For a $\zeta$-regular ideal point $\xi \in \partial_{\rm vis}\X$, there exists a unique $\tau=\zeta(\xi) \in \KPd$ such that every Weyl chamber containing $\xi$ also contains $\tau$, and likewise for $\iota \zeta$-regular points.

We will also be interested in a quantified version of regularity. 
The \emph{(first simple) root pseudometric} $d_\alpha(x,y)$ is given by $a_1-a_2$ when its vector-valued distance is $A = \Diag(a_1,\dots,a_d)$.
We note that, if $\sigma_i(x)$ denotes the $i$th singular value of $x \in \X$, then the root pseudodistance to the basepoint satisfies
$$ d_\alpha(I_d,x) = \frac{1}{2} \log \left( \frac{\sigma_1(x)}{\sigma_2(x)} \right) .$$

\subsection{Busemann functions centered at line-hyperplane pairs}

In the final step of the proof of Theorem \ref{thm:straight and spaced implies undistorted} we will make use of a comparison between the root pseudometric and certain Busemann functions.
Although the proof of Lemma \ref{lemma: dalpha and busemanns} is short, the idea behind it is clarified by a proper discussion about Busemann functions.

In general, if $c \colon [0,\infty) \to \X$ is a geodesic ray, then the associated \emph{Busemann function} $b \colon \X \to \R$ is defined by 
    \[ b(x) \coloneqq \lim_{t \to \infty} d(x,c(t)) - t . \]
The map taking geodesic rays to their corresponding Busemann functions induces a bijection of the visual boundary of $\X$ with the space of Busemann functions modulo constant functions.
The asymptote class of $c$ is called the \emph{center} of the Busemann function $b$.
Geometrically, the (sub)level sets of Busemann functions are limits of (balls, resp.)\ spheres which are neither points nor the whole space $\X$. 
In particular, in Euclidean space, Busemann functions are linear. 
When a flat $F$ in $\X$ contains a geodesic ray $c$, the induced Busemann function $b$ is linear on $F$. 

In a moment we will restrict attention to a specific $G$-orbit of Busemann functions. 
Recall that each maximal flat is congruent to the standard maximal flat
    \[ F_{std} \coloneqq \left\{
                \Diag(a_1,a_2,\dots,a_d)
                : a_1 a_2 \cdots a_d =1, a_i >0 \right\} ,\]
and that the roots of $F_{std}$ are given by $\alpha_{ij}(\exp(\Diag(a_1,\dots,a_d)).I_d) = a_i - a_j$.
Here we use the identification from \S \ref{subsec:scale of metric}.
$G$ acts transitively on the space of triples $(F,p,\alpha)$, where $F$ is a maximal flat, $p \in F$, and $\alpha \colon F \to \R$ is a root. 
The root $\alpha_{1d} \colon F_{std} \to \R$ is dual to the geodesic ray given by $c(t) = \exp(t w).I_d$ where
        \[ w=   \frac12
                \begin{bmatrix}  
                    1 &  &  &  &  \\
                    & 0 &  &  &  \\
                    & & \ddots & & \\
                    & & & 0 & \\                 
                    & & & & -1 \\
                \end{bmatrix} .\]
Simultaneous permutations of the rows and columns of $w$ produce the other coroots of $F_{std}$.

Now let $\mathcal{F}$ denote the space of \emph{line-hyperplane pairs}, i.e.\ 
    \[ \mathcal{F} \coloneqq \left\{ (\tau,\hat{\tau}) \in \KPd \times \KPdual : \tau \subset \hat{\tau} \right\} .\]
The map $g.([e_1],[e^d]) \mapsto [g.c]$ is a well-defined $G$-equivariant diffeomorphism from $\mathcal{F}$ onto its image in the visual boundary of $\X$.
Following the above discussion, its image is precisely the centers of Busemann functions that restrict to roots on any flat asymptotic to their centers. 
The Busemann function $b$ associated to $c$ is given by:
    \[ b(x) = \frac12 \log \left(\Abs{e_1}_{x^{-1}} \Abs{e^d}_{x}\right) .\] 




Recall from \cite[Section 2.5]{KLP17} that the \emph{star} $\st(\tau)$ of a simplex $\tau$ is the union of all Weyl chambers containing it and the \emph{Weyl cone} $V(o,\st(\tau))$ is the union of points on geodesic rays from $o \in \X$ to $\st(\tau)$.
When $\tau \in \KPd$, the intersection $\mathcal{F} \cap \st(\tau)$ consists of line-hyperplane pairs whose line agrees with $\tau$. 

\begin{lemma}[Root pseudometric and Busemann functions]\label{lemma: dalpha and busemanns}
    Let $\tau \in \KPd$. 
    If $y \in V(x,\st(\tau))$, then 
    $$ d_\alpha(x,y) = \min \{ b(x)-b(y) : b \in \mathcal{F} \cap \st(\tau) \} .$$
\end{lemma}

\begin{proof}
    Let $e_1,\dots,e_d$ denote the standard basis and $e^1,\dots,e^d$ denote the dual basis.
    Up to the action of $G=\SL(d,\K)$, we may assume that
    $$ x=I_d, \qquad b(x)-b(y) = - \frac12 \log \left(\Abs{e_1}_{y^{-1}} \Abs{e^d}_{y}\right), \quad \text{ and } y = 
    \begin{bmatrix}
    \lambda^{d-1} & 0 \\
    0 & \lambda^{-1}  A
    \end{bmatrix} $$
    for a Hermitian positive definite matrix $A$ of determinant $1$ with $\sigma_1(A) \le \lambda^d$.
    Then $b(x)-b(y) = \frac1{2} \log (\lambda^d/A_{dd}) \ge \frac12 \log (\lambda^d/\sigma_1(A)) = d_\alpha(x,y)$, and equality is achieved when $x,y$ and the center of $b$ lie in a common flat. 
\end{proof}

\subsection{Transversality and parallel sets}\label{sec:parallel sets}

A pair $\tau \in \KPd$ and $\hat{\tau} \in \KPdual$ are called \emph{transverse} (or \emph{antipodal} or \emph{opposite}) if $\tau + \hat{\tau} = \K^d$. 
When this occurs, there is a \emph{parallel set} which may be defined by 
$$ P(\hat{\tau},\tau) \coloneqq \{ p \in \X : S_p \tau = \hat{\tau} \} .$$
Equivalently, if $c$ is a geodesic with $c(+\infty) = \tau$ and $c(-\infty) = \hat{\tau}$ then $P(\hat{\tau},\tau)$ is the union of all points on geodesics parallel to $c$, equivalently, the union of all maximal flats containing the image of $c$.
When the points of $\X$ are interpreted as inner products, the parallel set $P(\hat{\tau},\tau)$ consists of those inner products making $\hat{\tau}$ perpendicular to $\tau$. 
A parallel set is a totally geodesic submanifold of $\X$. 

\begin{example}
    Let $e_1,\dots,e_d$ denote the standard basis.
    Then the parallel set associated to $\hat{\tau} = \mathrm{Span}(\{e_2,\dots,e_d\})$ and $\tau = \mathbb{R}e_1 $ is given by 
    $$ P( \hat{\tau},\tau) = \left\{ \begin{bmatrix}
    \lambda & 0 \\
    0 & A
    \end{bmatrix} \in \X \right\} .$$
\end{example}

\section{Angles and distances to parallel sets}

In this section we establish certain estimates to be used in the following section. 
The primary contribution is the angle-to-distance formula in Lemma \ref{lem:angle-distance formula}. 

\subsection{Angle-to-distance formula}

The following formula generalizes a familiar fact in real hyperbolic geometry. 
Consider a geodesic with endpoints $x,y$ in the visual boundary and a point $p$. 
Then the distance from $p$ to the geodesic $xy$ determines the angle at $p$ between $x$ and $y$, and vice versa. 

\begin{lemma}[Angle-to-distance formula]\label{lem:angle-distance formula}
    Let $q \in \X$, and let $\tau_+ \in \KP^{d-1}$, $\tau_- \in (\KP^{d-1})^\ast$ be transverse and let $P(\tau_-,\tau_+)$ denote the parallel set.
    Then
    $$ (d-1)\cos \measuredangle_q(\tau_-,\tau_+) + d \sech^2 \left( d(q,P(\tau_-,\tau_+) \right)=  1.$$ 
\end{lemma}

\begin{proof}
    Let $p \in P=P(\tau_-,\tau_+)$ and let $c(t)$ be a geodesic through $p$ normal to $P$. 
    Up to the action of $G = \SL(d,\K)$ and rescaling the speed of $c$, we may assume that $c(t)$ is given by $c(t)=g(t).p$ where 
    $$ g(t)^{-1}=\begin{bmatrix}  \cosh t & \sinh t &  &  &  \\
                        \sinh t & \cosh t &  &  &  \\
                        & & 1 & & \\
                        & & & \ddots & \\                       
                        & & & & 1 \\
                \end{bmatrix}.
    $$ 
    To compute the angle, we want to use that 
    $$ \measuredangle_{g(t)p}(\tau_-,\tau_+) = \measuredangle_{p}(g(t)^{-1}\tau_-,g(t)^{-1}\tau_+) = \measuredangle_{p}(k_-(t)\tau_-,k_+(t)\tau_+) $$
    where $g(t)^{-1}=k_-p_-$ according to $G=KG_{\tau_-}$ and $g(t)^{-1}=k_+p_+$ according to $G=KG_{\tau_+}$; this is called the Iwasawa or ``QR" decomposition.
    Explicitly, the top $2 \times 2$ blocks are given by:
        $$  \frac1{\sqrt{\cosh(2t)}}\begin{bmatrix}  
            \cosh t & -\sinh t \\
            \sinh t & \cosh t 
            \end{bmatrix}, \quad
            \frac1{\sqrt{\cosh(2t)}}\begin{bmatrix}  
            \cosh t & \sinh t \\
            -\sinh t & \cosh t 
            \end{bmatrix}, 
        $$ 
    for $k_+,k_-$ respectively.
    We have 
    $$ k_+^Tk_- = \begin{bmatrix}  
                         \sech(2t) & \tanh(2t) &  &  &  \\
                        -\tanh(2t) & \sech(2t) &  &  &  \\
                        & & 1 & & \\
                        & & & \ddots & \\                       
                        & & & & 1 \\
                    \end{bmatrix}, \text{ and } 
        k_-^Tk_+ = \begin{bmatrix}  
                        \sech(2t) & -\tanh(2t) &  &  &  \\
                        \tanh(2t) & \sech(2t) &  &  &  \\
                        & & 1 & & \\
                        & & & \ddots & \\                   
                        & & & & 1 \\
                    \end{bmatrix}, $$
    and for $Z = \Diag(d-1,-1,\dots,-1)$ the computation
    $$ \cos \measuredangle_p(k_-(t) \tau_-,k_+(t) \tau_+) = \frac{\Tr(k_- (-Z)k_-^Tk_+Zk_+^T)}{\Tr(Z^2)} = \frac1{d-1} \left( 1-d \sech^2(2t) \right) $$
    is straightforward.
    Our metric yields $d(p,c(t)) = d(p,g(t).p)=\sqrt{2 \Tr(X^2)}t =2t$ where $X$ is the unique symmetric matrix such that $g(t) = \exp(tX)$. 
\end{proof}

\begin{remark}\label{rem:no formula}
    Such a formula does not hold for parallel sets in symmetric spaces in general. 
    For example, if $2 \le k \le d-2$, $\tau \in \Gr(k,\K^d)$, and $\hat{\tau} \in \Gr(d-k,\K^d)$ is transverse to $\tau$, then the distance from $p \in \X$ to the parallel set $P(\hat{\tau},\tau)$ does not determine the angle $\measuredangle_p(\hat{\tau},\tau)$, and vice versa.
    Similarly, such a formula does not hold for maximal flats in $\X$ for any $d\ge 3$. 
\end{remark}

\begin{lemma}[Detecting transversality]\label{lem:angle to transversality}
    Let $q \in \X$, $\tau_+ \in \KPd$ and $\tau_- \in \KPdual$.
    Then $\tau_+$ is transverse to $\tau_-$ if and only if
    $$ \cos \measuredangle_q (\tau_-,\tau_+) < \frac1{d-1} .$$
\end{lemma}

In the proof we use the \emph{angular metric} $\measuredangle(x,y)$ on the visual boundary of $\X$, which may be defined as the supremum of all $\measuredangle_{p}(x,y)$ for $p \in \X$.\footnote{The length metric associated to the angular metric is known as the Tits metric \cite{BridsonHaefliger1999metricspaces,eberlein96}. When $d \ge 3$ the angular metric coincides with the Tits metric.}  
The angular metric is a $G$-invariant metric on the visual boundary of $\X$, but the topology it induces is much finer than the visual topology.
We note when a maximal flat contains $p$ and $x,y$ in its boundary, the angle between $x$ and $y$ is achieved at $p$. 

\begin{proof}
    If $\tau_+$ is transverse to $\tau_-$ then this follows from Lemma \ref{lem:angle-distance formula}.

    For the converse, observe that there are only two relative positions for a line and a hyperplane (either the line is in the hyperplane, or the line is transverse to the hyperplane). 
    In fact any angle between ideal points of these fixed types satisfies
    $$ \cos \measuredangle(\tau_-,\tau_+) \in \left\{ \frac1{d-1}, -1 \right\} .$$
    Now suppose that $ \cos \measuredangle_q (\tau_-,\tau_+) < \frac1{d-1},$ so the angle satisfies between $\tau_-$ and $\tau_+$ satisfies the same inequality.
    But then there exists a point $p \in \X$ whose Riemannian angle between $\tau_-$ and $\tau_+$ is $\pi$, so they are transverse. 
\end{proof}

\subsection{From rays to parallel sets}

In the next Lemma, we consider a parallel set and an asymptotic geodesic ray. 
Points far along the ray become exponentially close to the parallel set, with rate depending on the $d_\alpha$-pseudometric. 

\begin{lemma}[Distance from ray to parallel set]\label{lem:spacing to distance}
    Fix $D,S \ge 0$.
    Suppose $p,q \in \X$ satisfy $d_\alpha(p,q) \ge S$ and set $\tau = \zeta(pq)$. 
    Let $\hat{\tau}$ be transverse to $\tau$ and suppose the parallel set $P=P(\hat{\tau},\tau)$ satisfies $d(p,P) \le D$.
    Then 
    $$ d(q,P) \le \min \{ D, (e^D-1)e^{-S}\} .$$    
\end{lemma}

\begin{proof}
    The upper bound of $D$ is immediate from the convexity of the distance function. 

    The other bound is a mild improvement of \cite[Lemma 4.11]{riestenberg2024quantified}.
    We recall the setup there.
    For $\tau = \mathbb{R}e_1 \in \KP^{d-1}$, choose a basis $\{e_1,\dots,e_d\}$ containing $e_1$. 
    In this basis, the \emph{horocyclic subgroup} associated to $\tau$ is written as
        \[ N_\tau = \left\{ \begin{bmatrix}
                                1 & v^\dagger \\
                                0 & I_{d-1}
                                \end{bmatrix} : v \in \K^{d-1} \right\} \]
    where $I_{d-1}$ denotes the $(d-1) \times (d-1)$ identity matrix. 
    For $x \in \X$, the \emph{horocycle containing $p$ centered at $\tau$} is the orbit $H(x,\tau) = N_\tau .x$ of $\X$. 
    
    Let $r$ be the unique point on $P$ in the horocycle $H(p,\tau)$.
    Produce a horocyclic curve $r_0$ from $p$ to $r$, and then push it towards $\tau$ to obtain a horocyclic curve $r_\ell$ from $q$ to $P$.
    The proof of \cite[Lemma 4.11]{riestenberg2024quantified} provides the bound 
        \[\abs{\dot{r_\ell}(t)} \le e^{-S} \abs{\dot{r_0}(t)} \le e^{t-S} \]
    which we integrate over $t$ from $0$ to $D$ to estimate the length of $r_\ell$ and bound $d(q,P)$.
\end{proof}

\subsection{Controlling \texorpdfstring{$\zeta$}{zeta}-angles}

Kapovich-Leeb-Porti \cite{KLP14} introduced the \emph{$\zeta$-angle} which modifies the Riemannian angle by replacing regular directions with a direction in the same Weyl chamber but with type $\zeta$. 
We need a slight modification of their notion, because we work with a model simplex that is not invariant by the opposition involution $\iota$. 

\begin{definition}[$\zeta$-angle]\label{def:zeta angles}
    Let $x,y,z,w \in \X$ such that $xy$ is $\iota \zeta$-regular and $xz,xw$ are $\zeta$-regular and let $\hat{\tau} \in (\mathbb{KP}^{d-1})^\ast$ and $\tau \in \mathbb{KP}^{d-1}$.
    Let $\zeta(xz)$ (resp.\ $\zeta(xw)$) denote the unique ideal point of type $\zeta$ in a common Weyl chamber with $xz(+\infty)$ (resp.\ $xw(+\infty)$) and let $\iota\zeta(xy)$ denote the unique ideal point of type $\iota \zeta$ in a common Weyl chamber with $xy(+\infty)$.
    Then set:
    $$ \measuredangle^{\iota \zeta,\zeta}_{x}(y,z) \coloneqq \measuredangle_{x}(\iota\zeta(xy),\zeta(xz)) , \quad \measuredangle^{\zeta,\zeta}_{x}(z,w) \coloneqq \measuredangle_{x}(\zeta(xz),\zeta(xw)) ,$$
    $$ \measuredangle^{\iota \zeta,\zeta}_{x}(\hat{\tau},z) \coloneqq \measuredangle_{x}(\hat{\tau},\zeta(xz)) , \quad \measuredangle^{\iota\zeta,\zeta}_{x}(y,\tau) \coloneqq \measuredangle_{x}(\iota\zeta(xy),\tau) .$$
\end{definition}
See Figure \ref{fig:zeta directions}.

\begin{figure}
    \centering
    \begin{tikzpicture}
    \draw (-4,0) -- (4,0);
    \draw (-2,-3.46) -- (2,3.46);
    \draw (-2,3.46) -- (2,-3.46);
    \node[circle,scale=0.5,fill=black,label=below:$w$] (w) at (2.98,-0.3) {};
    \node[circle,scale=0.5,fill=black,label={[label distance=1mm]-10:$x$}] (x) at (0,0) {};
    \node[circle,scale=0.5,fill=black,label=left:$y$] (y) at (-1.98,-2.24) {};
    \node[circle,scale=0.5,fill=black,label=right:$z$] (z) at (2.6,1.5) {};
    \node[label=right:{$\zeta(xz)$}] (zxz) at (1.5,2.6) {};
    \node[label=right:{$\zeta(xw)$}] (zxw) at (1.5,-2.6) {};
    \node[scale=0.5,label=above:{$\iota\zeta(xz)=\iota\zeta(xw)$}] (izxz) at (3,0) {};
    \node[label=above:{$\zeta(xy)$}] (zxy) at (-3,0) {};
    \node[label=right:{$\iota\zeta(xy)$}] (izxy) at (-1.5,-2.6) {};
    \draw[-{latex[scale=2.5,length=2,width=3]},very thick] (x) -- (zxz);
    \draw[-{latex[scale=2.5,length=2,width=3]},very thick] (x) -- (izxz);
    \draw[-{latex[scale=2.5,length=2,width=3]},very thick] (x) -- (zxy);
    \draw[-{latex[scale=2.5,length=2,width=3]},very thick] (x) -- (izxy);
    \draw[-{latex[scale=2.5,length=2,width=3]},very thick] (x) -- (zxw);
    \end{tikzpicture}
    \caption{Points $w,x,y,z$ in a maximal flat and corresponding directions of types $\zeta$ and $\iota \zeta$.}
    \label{fig:zeta directions}
\end{figure}

\medskip

We conclude this section with an estimate controlling $\zeta$-angles between $x,y \in \X$ as seen from $o \in \X$ which is useful when the pseudodistances $d_\alpha(o,x),d_\alpha(o,y)$ are larger than $d(x,y)$.
Recall that $K = \SU(d,\K)$ is the stabilizer of $I_d$ in $G$.

\begin{lemma}[$\zeta$-angle estimate]\label{lem:lipschitz constant of zeta}
    Let $o \in \X$ and $S>D \ge0$. 
    The map 
    $$ f = \zeta(o\cdot) \colon \{x \in \X : d_\alpha(o,x) >0 \} \to (K\cdot Z,\measuredangle)$$ 
    is smooth, 
    and for $x,y \in \X$ such that $d_\alpha(o,x)\ge S$ and $d(x,y) \le D$ we have 
    $$ \measuredangle_o(\zeta(ox),\zeta(oy)) \le \sqrt{\frac{d}{2(d-1)}} \frac{d(x,y)}{\sinh(S-D)}.$$
\end{lemma}

\begin{proof}
    Since $G=\SL(d,\K)$ acts transitively by isometries we may assume that $o=I_d$.
    Now observe that the open Weyl cone
    $$ V = \left\{ q = \begin{bmatrix} \lambda & 0 \\ 0 & A \end{bmatrix} : \lambda > \sigma_1(A) \right\} $$
    is the preimage of $Z$ under $f$. 
    Since $V$ is an open subset of a parallel set, it is a smooth submanifold of $\X$.
    Let 
    $$ \mathfrak{k}^Z = \left\{ U = \begin{bmatrix} 0 & -u^\dagger \\ u & 0 \end{bmatrix} : u \in \mathbb{K}^{d-1} \right\} $$
    which is the Killing form perpendicular of $\mathfrak{k}_Z$, the infinitesimal stabilizer of $Z$ in $\mathfrak{k} = \mathfrak{su}(d,\K)$. 
    It is easy to check that for all $q$ in $V$ and $U \in \mathfrak{k}^Z$, the fundamental vector field $U^\ast_q$ is orthogonal to $T_qV$. 
    Mapping $K \times V$ by $(k,v)\mapsto kv$ descends to a $K$-equivariant diffeomorphism between $K \times_{K_Z} V$ and $X=\{q \in \X : d_\alpha(o,q) >0 \}$.
    It follows that $f$ is smooth and that $T_qV = \ker(df_q)$.

    We use that for a $C^1$ function between Riemannian manifolds, $f \colon X \to K\cdot Z$, the optimal Lipschitz constant is given by $\sup \left\{ \abs{ df_q} : q \in X \right\}$.
    Since the action of $K$ on $K\cdot Z$ is transitive and $f$ is equivariant, for any $v \in T_qX$ there exists $U \in \mathfrak{k}^{f(q)}$ such that $df_q(v)=U^\ast_{f(q)} = df_q(U^\ast_q)$.
    It is convenient to introduce the alternative notation $\mathrm{ev}_q(U) = U^\ast_q$ for the fundamental vector field (read ``evaluated at $q$").
    Since the decomposition $T_qX = \ker(df_q) \oplus \mathrm{ev}_q(\mathfrak{k}^{f(q)})$ is orthogonal, we have $\abs{v} \ge \abs{U^\ast_q}$, so 
    $$ \frac{\abs{df_q(v)}}{\abs{v}} \le \frac{\abs{U^\ast_{f(q)}}}{\abs{U^\ast_q}}.$$
    
    For a matrix $B$ let $\Abs{B}^2= 2\Tr(B^\dagger B)$. 
    For $U \in \mathfrak{k}^Z$ write $U = \sum_\beta U_\beta$ for its root space decomposition; this agrees with its decomposition into matrix entries, and note this decomposition is orthogonal with respect to $\Abs{\cdot}$.
    We recall that $Z$ is a diagonal matrix proportional to $\Diag(d-1,-1,\dots,-1)$ so positive roots which are nonzero on $Z$ all have the same value.
    We compute the norm squared of the vector $U^\ast_Z$ in the Euclidean space of Hermitian matrices:
    $$ \Abs{U^\ast_Z}^2 = \Abs{[Z,U]}^2 = \Abs{\sum_{\beta(Z)\ne 0} \beta(Z) U_\beta}^2 = \alpha(Z)^2 \Abs{U}^2 ,$$
    where the third inequality holds because $\beta(Z)=\alpha(Z)$ when $\beta(Z)>0$.
    Note that $\alpha(Z)^2 = \frac{d}{2(d-1)}$, see (\ref{eqn: Z}).
    For $q \in V$ satisfying $d_\alpha(o,q) \ge S'$ we compute the norm squared of the vector $U^\ast_q$ in $T_q\X$ with respect to the Riemannian metric:
    \begin{align*}
        \abs{U^\ast_q}^2_q & = \abs{ \mathrm{ev}_o \circ \Ad(q^{-1/2})(U)}^2_o && \text{since } q^{1/2}_\ast \circ \eva_{q^{1/2}.o} =  \eva_o \circ \Ad(q^{-1/2}) \\ 
        & = \abs{\mathrm{ev}_o \left(\sum_{\beta(Z)>0} e^{-\beta(\vec{d}(o,q))} U_\beta + e^{\beta(\vec{d}(o,q))} U_{-\beta}\right)}^2_o && \text{since } \Ad(q^{-1/2})(U_\beta) = e^{-\beta(\vec{d}(o,q))} U_\beta \\ 
        & = \abs{\sum_{\beta(Z)>0} \left( e^{-\beta(\vec{d}(o,q))} -e^{\beta(\vec{d}(o,q))}\right) (U_\beta)^\ast_o}^2_o && \text{since } U_\beta + U_{-\beta} \in \mathfrak{k} = \ker \eva_o \\ 
        & = \sum_{\beta(Z)>0} \left( e^{\beta(\vec{d}(o,q))} -e^{-\beta(\vec{d}(o,q))}\right)^2 \abs{(U_\beta)^\ast_o}^2_o && \text{since the root spaces } \mathfrak{g}_\beta \text{ are orthogonal} \\
        & = \sum_{\beta(Z)>0} 4 \sinh^2(\beta(\vec{d}(o,q))) \abs{(U_\beta)^\ast_o}^2_o \\ 
        & \ge \sinh^2(S') \cdot 2 \sum_{\beta(Z)>0} 2 \abs{(U_\beta)^\ast_o}^2_o && \text{by assumption that } d_\alpha(o,q) \ge S' \\ 
        & = \sinh^2(S') \cdot 2 \sum_{\beta(Z)> 0} \Abs{U_\beta}^2 && \text{since } 2 \abs{(U_\beta)^\ast_o}^2 = \Abs{U_\beta}^2 \\ 
        & = \sinh^2(S') \Abs{U}^2 && \text{since } \Abs{U_\beta} = \Abs{U_{-\beta}} \text{ for } \beta(Z)\ne 0 . 
    \end{align*}
    To conclude we observe that the geodesic segment from $x$ to $y$ in $\X$ lies in 
    $\{q \in \X : d_\alpha(o,q) \ge S-D \}$.
\end{proof}

\section{Straight and spaced sequences}\label{sec: straight and spaced}

Kapovich-Leeb-Porti \cite{KLP17,KLP18a} and Bochi-Potrie-Sambarino \cite{BPS19} have proven that a finitely generated subgroup $\Gamma$ of $\SL(d,\K)$ is projective Anosov if and only if every geodesic in $\Gamma$ maps to a uniformly $d_\alpha$-undistorted sequence in $\X$, see Definition \ref{def:definitions for sequences}(1). 
In this section we state and prove a local criterion for a sequence in $\X$ to be globally and uniformly $d_\alpha$-undistorted. 

\medskip

Theorem \ref{thm:straight and spaced implies undistorted} is similar to
\cite[Theorem 7.2]{KLP14},
\cite[Theorem 3.18]{KapovichLeebPorti2023localtoglobal} and \cite[Theorem 5.1]{riestenberg2024quantified} but subtly different.
The statement here assumes that the sequence is $S$-spaced with respect to the $d_\alpha$-pseudometric, which is a bit weaker than assuming that consecutive pairs are simultaneously uniformly regular and $S$-spaced with respect to the Riemannian metric.
The conclusion is also weaker, since we only obtain that the sequence is $d_\alpha$-undistorted, and it may fail to fellow travel truncated Weyl cones.
However, once this statement, which is purely about sequences in the symmetric space, is applied to actions of finitely generated subgroups where the sequences come from geodesics in the Cayley graph, the Lipschitz property of the orbit map implies the sequences are Morse. 
It is straightforward to modify the proof of Theorem \ref{thm:straight and spaced implies undistorted} to obtain a statement which assumes this stronger condition and implies that the sequence is uniformly Morse. 

\medskip

We need the following definitions in order to state the main theorem.
Recall the pseudometric $d_\alpha$ from \S \ref{sec:regularity and pseudometric} and the $\zeta$-angle from Definition \ref{def:zeta angles}.

\begin{definition}\label{def:definitions for sequences}
    Let $(x_n)$ be a sequence of points in $\X$.

    \begin{enumerate}
        \item The sequence is \emph{$d_\alpha$-undistorted with constants $c_1>0$ and $c_2 \ge 0$} if for all $m,n$:
        $$ c_1\abs{m-n} -c_2 \le d_\alpha(x_n,x_m) .$$
        \item The sequence is \emph{$S$-spaced} for $s\ge 0$ if for all $n$:
        $$ d_\alpha(x_n,x_{n+1}) \ge S .$$
        \item The sequence is \emph{$\epsilon$-straight} if each segment $x_nx_{n+1}$ is $\zeta$-regular and for all $n$:
        $$ \measuredangle^{\iota \zeta,\zeta}_{x_n}(x_{n-1},x_{n+1})\ge \pi-\epsilon .$$
    \end{enumerate}
\end{definition}

We may now state and prove the main theorem: sufficiently straight and spaced sequences are $d_\alpha$-undistorted.
The inequalities appearing in the hypotheses are briefly explained in Remark \ref{rem:assumptions explained}.

\begin{theorem}[Sufficiently straight and spaced sequences are $d_\alpha$-undistorted]\label{thm:straight and spaced implies undistorted}
    
    Let $\epsilon_{aux}>\epsilon$, $S$ and $\delta_1,\delta_2,\delta_3,\delta_4$ satisfy:
    
    \begin{equation}\label{ass:antipodal}
    \min \left\{ \cos\left(\epsilon_{aux} + \sqrt{\frac{d}{2(d-1)}} \frac{\delta_4}{\sinh(S-\delta_4)} \right) , \cos\left(2\epsilon_{aux}-\epsilon \right)\right\} > - \frac1{d-1} ,
    \end{equation}
    \begin{equation}\label{ass:angle to distance}
        \max \left\{ (1-d) \cos(\epsilon_{aux})+d \sech^2(\delta_1) , (1-d) \cos(2\epsilon_{aux}-\epsilon)+d \sech^2(\delta_3) \right\} \le 1 ,
    \end{equation}
    \begin{equation}\label{ass:distance to angle}
        (1-d) \cos(\epsilon_{aux}-\epsilon)+d \sech^2(\delta_2) \le 1 ,
    \end{equation}
    \begin{equation}\label{ass:spaced-to-distance}
        (e^{\delta_1}-1)e^{-S} \le \delta_2, 
    \end{equation}
    \begin{equation}\label{ass:sufficient spacing}
        \delta_4 \ge \min \{2 \delta_3, \delta_3 + (e^{\delta_3}-1)e^{-S} \} \text{ and } S > 2 \delta_4 .
    \end{equation}

    Then an $S$-spaced and $\epsilon$-straight sequence is $d_\alpha$-undistorted with constants $(S-2\delta_4,2\delta_4)$.
\end{theorem}

\begin{remark}\label{rem:assumptions explained}
    Assumption \ref{ass:antipodal} guarantees that certain pairs of simplices arising in the proof are antipodal by Lemma \ref{lem:angle to transversality}.
    Assumption \ref{ass:angle to distance} (resp.\ Assumption \ref{ass:distance to angle}) allows us to apply Lemma \ref{lem:angle-distance formula} to see that  certain points are close to parallel sets (resp.\ certain $\zeta$-angles are small enough).
    We use 
    Assumption \ref{ass:spaced-to-distance} to apply Lemma \ref{lem:spacing to distance} and find that certain points are close to parallel sets.
    Finally we use 
    Assumption \ref{ass:sufficient spacing} to obtain uniformly monotonic projections to Weyl cones.
\end{remark}

\begin{remark}
    The reader may want to convince themself that for $\epsilon$ sufficiently small and $S$ sufficiently large, there exist auxiliary parameters satisfying the hypotheses of Theorem \ref{thm:straight and spaced implies undistorted}.
    To see this, choose any $\epsilon$ smaller than $\epsilon_{max} = \cos^{-1} \left( \frac{-1}{d-1} \right)$ and any $\epsilon_{aux}$ satisfying $\epsilon < \epsilon_{aux} < \frac{\epsilon+\epsilon_{max}}{2}$.
    Then for any $\delta_1,\delta_2,\delta_3$ satisfying Assumptions \ref{ass:angle to distance} and \ref{ass:distance to angle} and any $\delta_4 \ge 2 \delta_3$, one observes that for $S$ sufficiently large, Assumptions \ref{ass:antipodal}, \ref{ass:spaced-to-distance} and \ref{ass:sufficient spacing} are satisfied.
\end{remark}

\begin{proof}
    Step 1: Propagation. 
    We first need to show that under these assumptions, the property of ``moving $\epsilon_{aux}$-away from/towards a hyperplane/line" propagates along the sequence.
    Assume that $\hat{\tau} \in \KPdual$ satisfies 
    $$ \measuredangle^{\iota \zeta,\zeta}_{x_0}(\hat{\tau},x_1) \ge \pi- \epsilon_{aux} .$$
    Write $\tau_{01} \in \KPd$ for $\zeta(x_0x_1(+\infty))$.
    By Assumption \ref{ass:antipodal} and Lemma \ref{lem:angle to transversality}, we have that $\hat{\tau}$ is transverse to $\tau_{01}$.
    By Assumption \ref{ass:angle to distance} and Lemma \ref{lem:angle-distance formula}, we have
    $$ d(x_0,P(\hat{\tau},\tau_{01})) \le \sech^{-1}\left(\sqrt{\frac1{d}\left(1-(1-d)\cos (\epsilon_{aux})\right)}\right) \le \delta_1. $$
    By Assumption \ref{ass:spaced-to-distance} and Lemma \ref{lem:spacing to distance}, we have 
    $$ d(x_1,P(\hat{\tau},\tau_{01})) \le  (e^{\delta_1}-1)e^{-S} \le \delta_2 $$
    small enough that Assumption \ref{ass:distance to angle} and Lemma \ref{lem:angle-distance formula} imply 
    $$ \measuredangle^{\iota \zeta,\iota \zeta}_{x_1}(\hat{\tau},x_0) \le \pi-\cos^{-1}\left( \frac1{d-1} \left( 1-d \sech^2 \left(\delta_2 \right) \right) \right) \le \epsilon_{aux}-\epsilon. $$
    By straightness and the triangle inequality, we have 
    $$ \measuredangle^{\iota\zeta,\zeta}_{x_1}(\hat{\tau},x_2) \ge \pi - \epsilon_{aux} .$$
    By induction, we have that $\measuredangle^{\iota\zeta,\zeta}_{x_n}(\hat{\tau},x_{n+1})\ge \pi- \epsilon_{aux}$ for all $n \ge 0$.
    A similar proof shows that the property of moving $\epsilon_{aux}$-towards a line propagates along the sequence.

    \medskip

    Step 2: Extraction. 
    The previous step allows us to find simplices $\tau_\pm$ that the sequence moves $\epsilon_{aux}$-towards/away from.
    Indeed, let 
    $$ C^+_n = \{ \tau_+ : \measuredangle^{\iota\zeta,\zeta}_{x_n}(x_{n-1},\tau_+) \ge \pi-\epsilon_{aux} \}  \text{ and } C^-_n = \{ \tau_- : \measuredangle^{\iota\zeta,\zeta}_{x_n}(\tau_-,x_{n+1}) \ge \pi-\epsilon_{aux} \}.$$
    The previous step implies that $\bigcap_{n} C^\pm_n$ is nonempty, so we may extract $\tau_\pm \in \bigcap_n C^\pm_n$.
    The proof of the previous step shows that since $\tau_- \in C^-_{n-1}$ we have  
    $$ \measuredangle^{\iota \zeta,\zeta}_{x_{n-1}}(\tau_-,x_n) \ge \pi- \epsilon_{aux} \implies \measuredangle^{\iota \zeta,\iota \zeta}_{x_n}(\tau_-,x_{n-1}) \le \epsilon_{aux}-\epsilon $$ 
    and a similar statement holds for $\tau_+$, so the triangle inequality gives
    $$ \abs{\measuredangle^{\iota \zeta,\zeta}_{x_n}(\tau_-,\tau_+) - \measuredangle^{\iota \zeta,\zeta}_{x_n}(x_{n-1},x_{n+1})} \le \measuredangle^{\iota \zeta,\iota \zeta}_{x_n}(\tau_-,x_{n-1}) + \measuredangle^{\zeta,\zeta}_{x_n}(\tau_+,x_{n+1}) \le 2 (\epsilon_{aux}-\epsilon) .$$
    By straightness, the previous inequality implies that $\measuredangle^{\iota \zeta,\zeta}_{x_n}(\tau_-,\tau_+) \ge \pi - 2\epsilon_{aux} + \epsilon$, so $\tau_-$ is transverse to $\tau_+$ by Assumption \ref{ass:antipodal} and Lemma \ref{lem:angle to transversality}.
    Moreover, 
    $$ d(x_n,P(\tau_-,\tau_+)) \le \sech^{-1}\left(\sqrt{\frac1{d}\left(1-(1-d)\cos (2 \epsilon_{aux}-\epsilon)\right)}\right) \le \delta_3 $$
    by Assumption \ref{ass:angle to distance} and Lemma \ref{lem:angle-distance formula}.

    \medskip
    
    Step 3: Undistortion. 
    We now verify that the sequence $(x_n)$ is $d_\alpha$-undistorted. 
    For all $n$, let $\overline{x}_n$ denote the nearest point to $x_n$ in the parallel set $P(\tau_-,\tau_+)$.
    We will show that the sequence $(\overline{x}_n)$ is $d_\alpha$-undistorted. 
    
    Let $\xi$ be the ideal point corresponding to the ray $\overline{x}_n\overline{x}_{n+1}$. 
    Since the rays $x_n\xi$ and $\overline{x}_n\xi$ are asymptotic, their Hausdorff distance is at most $d(x_n,\overline{x}_n)\le \delta_3$,
    so, by the proof of Lemma \ref{lem:spacing to distance},  $x_{n+1}$ is at most $\delta_4$ from $x_n\xi$. 
    We then have
    \begin{equation*}\label{eqn:nested}
        \measuredangle^{\iota \zeta,\zeta}(\tau_-,\xi) \ge \measuredangle_{x_n}^{\iota \zeta,\zeta}(\tau_-,\xi) \ge \measuredangle_{x_n}^{\iota \zeta,\zeta}(\tau_-,x_{n+1}) - \measuredangle_{x_n}^{\zeta,\zeta}(x_{n+1},\xi)  \ge \pi - \epsilon_{aux} - \measuredangle_{x_n}^{\zeta,\zeta}(x_{n+1},\xi)
    \end{equation*}
    and we can bound 
    $$ \measuredangle^{\zeta,\zeta}_{x_n}(x_{n+1},\xi) \le \sqrt{\frac{d}{2(d-1)}} \frac{\delta_4}{\sinh(S-\delta_4)}$$ 
    by Lemma \ref{lem:lipschitz constant of zeta}. 

    By Assumption \ref{ass:antipodal}, it follows that $\tau_-$ is antipodal to $\zeta(\xi)$; since $\tau_+$ is the only simplex in $\partial_{\rm vis}P(\tau_-,\tau_+)$ antipodal to $\tau_-$, this implies $\tau(\xi)=\tau_+$.
    By the convexity of Weyl cones the sequence of projections $\overline{x}_n$ land in nested Weyl cones in $P(\tau_-,\tau_+)$ \cite[Corollary 2.11]{KLP17}.

    Finally we show that $(\overline{x}_n)$ is $d_\alpha$-undistorted.
    Note that the vector-valued triangle inequality (see \cite{KLP17,Parreau} or \cite[Corollary 3.11]{riestenberg2024quantified}) implies that $d_\alpha(\overline{x}_{n},\overline{x}_{n+1}) \ge d_\alpha(x_{n},x_{n+1}) - d(x_n,\overline{x}_n) - d(x_{n+1},\overline{x}_{n+1}) \ge S-2\delta_4$, which is positive by Assumption \ref{ass:sufficient spacing}.
    Fix $m>n$. 
    By Lemma \ref{lemma: dalpha and busemanns} and the nestedness of Weyl cones, 
    there exists a Busemann function $b$ so that 
    \begin{align*}
        d_\alpha(\overline{x}_n,\overline{x}_m) & = b(\overline{x}_n)-b(\overline{x}_m) = b(\overline{x}_n)-b(\overline{x}_{n+1})+b(\overline{x}_{n+1}) - b(\overline{x}_{n+2}) +\cdots + b(\overline{x}_{m-1}) - b(\overline{x}_{m}) \\
        & \ge d_\alpha(\overline{x}_{n},\overline{x}_{n+1}) + d_\alpha(\overline{x}_{n+1},\overline{x}_{n+2}) +\cdots + d_\alpha(\overline{x}_{m-1},\overline{x}_{m}) \ge (m-n)(S-2\delta_4).
    \end{align*}
    This implies that $(x_n)$ is $d_\alpha$-undistorted with constants $(S-2\delta_4,2\delta_4)$.
\end{proof}

\section{The algorithm by example}

We first briefly describe the algorithm to certify a subgroup is Anosov.
If a finitely generated subgroup $\Gamma$ of $\SL(d,\mathbb{K})$ is Anosov, then eventually the orbit map takes geodesics in $\Gamma$ to straight-and-spaced sequences in $\X$, up to coarsifying and passing to midpoints, see \cite[Proposition 7.16]{KLP14} or \cite[Theorem 5.5]{riestenberg2024quantified}. 
The algorithm to certify the Anosov property is then to check the straight-and-spaced condition on a finite ball in the Cayley graph of $\Gamma$. 
If the check passes, then the subgroup is Anosov; otherwise, we examine words of larger length and check the straight-and-spaced condition again.
The algorithm terminates if and only if $\Gamma$ is Anosov.
We remark that this process is slightly different than the process described in \cite[Section 7.7]{KLP14}, in that we check the straight-and-spaced condition instead of their local Morse condition. 

We illustrate the algorithm by an example.
We consider a fixed surface subgroup of $\SL(3,\R)$ and verify the Anosov property via a computation involving only the words of length $8$. 
Using KBMAG \cite{kbmag1.5.11}, the words of length $8$ can be enumerated via a finite state automoton. 
Then via \cite{Rie24code} or \cite{Weis24code}, the straightness and spacing parameters of the associated midpoint sequence can be computed.
It is then easy to obtain auxiliary parameters satisfying Theorem \ref{thm:straight and spaced implies undistorted}, which certifies the Anosov condition.

\medskip

Let $\Gamma$ be the subgroup of $\SL(3,\R)$ generated by 
$$ S = \left\{ 
\begin{bmatrix}
\cos(\theta) & \sin(\theta) & 0 \\
-\sin(\theta) & \cos(\theta) & 0 \\
0 & 0 & 1 
\end{bmatrix}
\begin{bmatrix}
\cosh(T) & 0 & \sinh(T) \\
0 & 1 & 0 \\
\sinh(T) & 0 & \cosh(T) 
\end{bmatrix}
\begin{bmatrix}
\cos(\theta) & -\sin(\theta) & 0 \\
\sin(\theta) & \cos(\theta) & 0 \\
0 & 0 & 1 
\end{bmatrix}
: \theta \in \left\{0,\frac{\pi}{8},\frac{\pi}{4},\frac{3\pi}{8}\right\} \right\} $$
where $T=2\cosh^{-1}(\cot(\pi/8))$.
In fact, $\Gamma$ preserves a symmetric bilinear form of signature $(2,1)$ and hence acts naturally on the hyperbolic plane. 
The corresponding quotient is the Bolza surface.

Labelling the generators by $a,b,c,d$ and their inverses by $A,B,C,D$, a presentation of $\Gamma$ is given by the single relation $adCbADcB$.
The words of length at most $8$ in $\Gamma$ can be enumerated using the KBMAG package in GAP \cite{kbmag1.5.11}.
For each such word $w$, we decompose it as $w=w_1w_2$ into words of length $4$. 

We set 
$$ m_1= \midp(o,w_1^{-1}o), \text{ and } m_2 = \midp(o,w_2o) $$
and compute 
$$ s_w = d_\alpha(m_1,m_2), \quad \epsilon_w^+ = \measuredangle_{m_1}^{\zeta,\zeta}(o,m_2), \text{ and } \epsilon_w^- = \measuredangle_{m_2}^{\iota \zeta,\iota \zeta}(o,m_1).$$

We now consider a geodesic $(\gamma_n)$ in $\Gamma$ with $\gamma_0=id$. 
We will see that $\gamma_n o$ is $d_\alpha$-undistorted sequence in $\X$, equivalently, that $\gamma_{4n} o$ is $d_\alpha$-undistorted, equivalently, that the sequence of midpoints $m_n=\midp(\gamma_{4n} o,\gamma_{4n+4} o)$ is $d_\alpha$-undistorted.
In general, if a finitely generated subgroup $\Gamma$ is Anosov then for some $k>0$ the sequence of midpoints $m_n=\midp(\gamma_{kn} o,\gamma_{kn+k} o)$ is straight-and-spaced, which can be verified by examining words of length $3k$, see \cite[Proposition 7.16]{KLP14} or \cite[Theorem 5.5]{riestenberg2024quantified}. 
We use the same trick as \cite{Rie23verifying} to instead consider words of length $2k$, where $k=4$ in this example. 
The idea is that one can estimate the straightness parameter for words of length $3k$ with words of length $2k$. 

Since $\Gamma$ acts by isometries, the sequence of midpoints is $S$-spaced for 
$$ S = \min \{s_w :\abs{w}=8 \} ,$$
and $\epsilon$-straight for 
$$ \epsilon = \max\{\epsilon_w^+ : \abs{w}=8 \} + \max\{\epsilon_w^- : \abs{w}=8 \} .$$ 

We compute these in the example in the code available at \cite{Rie24code}, obtaining 
$$ \min\{\cos(\epsilon_w^+) : \abs{w}=8 \}, \min\{\cos(\epsilon_w^-) : \abs{w}=8 \} \approx 0.87 \implies \epsilon \approx 2\cos^{-1}(0.87) \approx 1.03, $$
$$ S \approx 3.08 .$$ 

We choose auxiliary constants $\epsilon_{aux} = 0.7 \epsilon + 0.3 \epsilon_{max}$, where $\epsilon_{max} = \cos^{-1} \left( \frac{-1}{3-1} \right)$, and
$$ \delta_1 = \sech^{-1} \left(\sqrt{\frac1{3}(1-(1-3)\cos(\epsilon_{aux}))}\right) \approx 0.92 $$
$$ \delta_2 = \sech^{-1} \left(\sqrt{\frac1{3}(1-(1-3)\cos(\epsilon_{aux}-\epsilon))}\right) \approx 0.18 $$
$$ \delta_3 = \sech^{-1} \left(\sqrt{\frac1{3}(1-(1-3)\cos(2\epsilon_{aux}-\epsilon))}\right) \approx 1.29 $$
$$ \delta_4 = \delta_3 + (e^{\delta_3}-1)e^{-S} \approx 1.41 .$$
These constants satisfy the conditions of Theorem \ref{thm:straight and spaced implies undistorted}.
So each sequence of midpoints $(m_n)$ constructed above is $d_\alpha$-undistorted.
This implies that every geodesic in $\Gamma$ maps to a $d_\alpha$-undistorted sequence in $\X$. 
By \cite{KLP18b} or \cite{BPS19}, it follows that $\Gamma$ is Anosov.

\bibliographystyle{amsalpha}
\bibliography{biblio0,biblio}

@article {KapovichLeebPorti2023localtoglobal,
    AUTHOR = {Kapovich, Michael and Leeb, Bernhard and Porti, Joan},
     TITLE = {Morse actions of discrete groups on symmetric spaces:
              local-to-global principle},
   JOURNAL = {Geom. Topol.},
  FJOURNAL = {Geometry \& Topology},
    VOLUME = {29},
      YEAR = {2025},
    NUMBER = {5},
     PAGES = {2343--2390},
      ISSN = {1465-3060,1364-0380},
   MRCLASS = {22E40 (20F65 53C35)},
  MRNUMBER = {4950918},
       DOI = {10.2140/gt.2025.29.2343},
       URL = {https://doi.org/10.2140/gt.2025.29.2343},
}

@misc{kbmag1.5.11,   
author =           {Holt, D. and GAP Team, T.},   
title =            {{kbmag}, {K}nuth-{B}endix on {M}onoids and {A}utomatic {G}roups,                       {V}ersion 1.5.11},   
month =            {Jan},   
year =             {2023},   
note =             {Refereed GAP package},   
howpublished =     {\href           {https://gap-packages.github.io/kbmag}                       {\texttt{https://gap-packages.github.io/}\discretionary                       {}{}{}\texttt{kbmag}}},   
printedkey =       {HT23} }

@misc{davaloriestenberg2024dirichletanosov,
      title={Finite-sided {D}irichlet domains and {A}nosov subgroups}, 
      author={Colin Davalo and J. Maxwell Riestenberg},
      year={2024},
      eprint={2402.06408},
      archivePrefix={arXiv},
      primaryClass={math.GT},
      url={https://arxiv.org/abs/2402.06408}, 
}

@article {Lab06,
    AUTHOR = {Labourie, Fran\c{c}ois},
     TITLE = {Anosov flows, surface groups and curves in projective space},
   JOURNAL = {Invent. Math.},
  FJOURNAL = {Inventiones Mathematicae},
    VOLUME = {165},
      YEAR = {2006},
    NUMBER = {1},
     PAGES = {51--114},
      ISSN = {0020-9910},
   MRCLASS = {20F65 (37D20 37F30)},
  MRNUMBER = {2221137},
MRREVIEWER = {Richard Kenyon},
       DOI = {10.1007/s00222-005-0487-3},
       URL = {https://doi.org/10.1007/s00222-005-0487-3},
}

@article {GW12,
    AUTHOR = {Guichard, Olivier and Wienhard, Anna},
     TITLE = {Anosov representations: domains of discontinuity and
              applications},
   JOURNAL = {Invent. Math.},
  FJOURNAL = {Inventiones Mathematicae},
    VOLUME = {190},
      YEAR = {2012},
    NUMBER = {2},
     PAGES = {357--438},
      ISSN = {0020-9910},
   MRCLASS = {22F30 (32G15 53C30 53D25)},
  MRNUMBER = {2981818},
MRREVIEWER = {Pablo Su\'{a}rez-Serrato},
       DOI = {10.1007/s00222-012-0382-7},
       URL = {https://doi.org/10.1007/s00222-012-0382-7},
}

@article {KLP18a,
    AUTHOR = {Kapovich, Michael and Leeb, Bernhard and Porti, Joan},
     TITLE = {Dynamics on flag manifolds: domains of proper discontinuity
              and cocompactness},
   JOURNAL = {Geom. Topol.},
  FJOURNAL = {Geometry \& Topology},
    VOLUME = {22},
      YEAR = {2018},
    NUMBER = {1},
     PAGES = {157--234},
      ISSN = {1465-3060},
   MRCLASS = {53C35 (22E40 37B05 51E24)},
  MRNUMBER = {3720343},
MRREVIEWER = {Gabriele Link},
       DOI = {10.2140/gt.2018.22.157},
       URL = {https://doi.org/10.2140/gt.2018.22.157},
}

@misc{KLP14,
      title={Morse actions of discrete groups on symmetric space}, 
      author={Michael Kapovich and Bernhard Leeb and Joan Porti},
      year={2014},
      eprint={1403.7671},
      archivePrefix={arXiv},
      primaryClass={math.GR},
      url={https://arxiv.org/abs/1403.7671}, 
}

@book{eberlein96,
  title={Geometry of nonpositively curved manifolds},
  author={Eberlein, Patrick},
  year={1996},
  publisher={University of Chicago Press}
}

@book{helgason79,
  title={Differential geometry, {L}ie groups, and symmetric spaces},
  author={Helgason, Sigurdur},
  volume={80},
  year={1979},
  publisher={Academic press}
}

@article{KLP18b,
  title={A {M}orse Lemma for quasigeodesics in symmetric spaces and euclidean buildings},
  author={Kapovich, Michael and Leeb, Bernhard and Porti, Joan},
  journal={Geometry \& Topology},
  volume={22},
  number={7},
  pages={3827--3923},
  year={2018},
  publisher={Mathematical Sciences Publishers}
}

@article{KLP17,
  title={Anosov subgroups: dynamical and geometric characterizations},
  author={Kapovich, Michael and Leeb, Bernhard and Porti, Joan},
  journal={European Journal of Mathematics},
  volume={3},
  number={4},
  pages={808--898},
  year={2017},
  publisher={Springer}
}

@book {BridsonHaefliger1999metricspaces,
	AUTHOR = {Bridson, Martin R. and Haefliger, Andr\'{e}},
	TITLE = {Metric spaces of non-positive curvature},
	SERIES = {Grundlehren der Mathematischen Wissenschaften [Fundamental
	Principles of Mathematical Sciences]},
	VOLUME = {319},
	PUBLISHER = {Springer-Verlag, Berlin},
	YEAR = {1999},
	PAGES = {xxii+643},
	ISBN = {3-540-64324-9},
	MRCLASS = {53C23 (20F65 53C70 57M07)},
	MRNUMBER = {1744486},
	MRREVIEWER = {Athanase Papadopoulos},
	DOI = {10.1007/978-3-662-12494-9},
	URL = {https://doi.org/10.1007/978-3-662-12494-9},
}

@unpublished{Parreau,
author = {Parreau, Anne},
title = {La distance vectorielle dans les immeubles affines et les espaces symétriques},
note= {In preparation},
}

@article {BPS19,
    AUTHOR = {Bochi, Jairo and Potrie, Rafael and Sambarino, Andr\'{e}s},
     TITLE = {Anosov representations and dominated splittings},
   JOURNAL = {J. Eur. Math. Soc. (JEMS)},
  FJOURNAL = {Journal of the European Mathematical Society (JEMS)},
    VOLUME = {21},
      YEAR = {2019},
    NUMBER = {11},
     PAGES = {3343--3414},
      ISSN = {1435-9855},
   MRCLASS = {22E40 (20F67 37B99 37D30 53C35)},
  MRNUMBER = {4012341},
       DOI = {10.4171/JEMS/905},
       URL = {https://doi-org.ezproxy.lib.utexas.edu/10.4171/JEMS/905},
}

@article{kapovich2016discreteness,
  title={Discreteness is undecidable},
  author={Kapovich, Michael},
  journal={International Journal of Algebra and Computation},
  volume={26},
  number={03},
  pages={467--472},
  year={2016}
}

@inproceedings {Kas18,
    AUTHOR = {Kassel, Fanny},
     TITLE = {Geometric structures and representations of discrete groups},
 BOOKTITLE = {Proceedings of the {I}nternational {C}ongress of
              {M}athematicians---{R}io de {J}aneiro 2018. {V}ol. {II}.
              {I}nvited lectures},
     PAGES = {1115--1151},
 PUBLISHER = {World Sci. Publ., Hackensack, NJ},
      YEAR = {2018},
   MRCLASS = {57N16 (20H10 22E40 57M50 57S30)},
  MRNUMBER = {3966802},
}

@incollection {GW18,
    AUTHOR = {Guichard, Olivier and Wienhard, Anna},
     TITLE = {Positivity and higher {T}eichm\"{u}ller theory},
 BOOKTITLE = {European {C}ongress of {M}athematics},
     PAGES = {289--310},
 PUBLISHER = {Eur. Math. Soc., Z\"{u}rich},
      YEAR = {2018},
   MRCLASS = {22E15 (22E46)},
  MRNUMBER = {3887772},
MRREVIEWER = {Renato G. Bettiol},
}

@article{Hit92,
	title = {Lie groups and {T}eichmüller space},
	journal = {Topology},
	volume = {31},
	number = {3},
	pages = {449-473},
	year = {1992},
	issn = {0040-9383},
	doi = {https://doi.org/10.1016/0040-9383(92)90044-I},
	url = {https://www.sciencedirect.com/science/article/pii/004093839290044I},
	author = {N.J. Hitchin}
}

@incollection {Kap23,
	AUTHOR = {Kapovich, Michael},
	TITLE = {Geometric algorithms for discreteness and faithfulness},
	BOOKTITLE = {Computational aspects of discrete subgroups of {L}ie groups},
	SERIES = {Contemp. Math.},
	VOLUME = {783},
	PAGES = {87--112},
	PUBLISHER = {Amer. Math. Soc., [Providence], RI},
	YEAR = {[2023] \copyright 2023},
	ISBN = {978-1-4704-6804-0},
	MRCLASS = {20-08 (20F67 22E40 53C35)},
	MRNUMBER = {4556436},
	DOI = {10.1090/conm/783/15736},
	URL = {https://doi.org/10.1090/conm/783/15736},
}

@article{fockModuliSpacesLocal2006a,
  title = {Moduli Spaces of Local Systems and Higher {{Teichm{\"u}ller}} Theory},
  author = {Fock, Vladimir and Goncharov, Alexander},
  year = {2006},
  journal = {Publications Mathematiques de l'Institut des Hautes Etudes Scientifiques},
  volume = {103},
  number = {1},
  pages = {1--211},
  doi = {10.1007/s10240-006-0039-4}
}

@incollection {Rie23verifying,
    AUTHOR = {Riestenberg, J. Maxwell},
     TITLE = {Verifying the straight-and-spaced condition},
 BOOKTITLE = {Computational {A}spects of {D}iscrete {S}ubgroups of {L}ie
              {G}roups},
    SERIES = {Contemp. Math.},
    VOLUME = {783},
     PAGES = {135--140},
 PUBLISHER = {Amer. Math. Soc., [Providence], RI},
      YEAR = {[2023] \copyright2023},
      ISBN = {978-1-4704-6804-0},
   MRCLASS = {53C35 (20F65 22E40)},
  MRNUMBER = {4556439},
MRREVIEWER = {Salah\ Mehdi},
       DOI = {10.1090/conm/783/15700},
       URL = {https://doi.org/10.1090/conm/783/15700},
}

@article {GM1991algorithm,
    AUTHOR = {Gilman, J. and Maskit, B.},
     TITLE = {An algorithm for {$2$}-generator {F}uchsian groups},
   JOURNAL = {Michigan Math. J.},
  FJOURNAL = {Michigan Mathematical Journal},
    VOLUME = {38},
      YEAR = {1991},
    NUMBER = {1},
     PAGES = {13--32},
      ISSN = {0026-2285,1945-2365},
   MRCLASS = {30F35 (20H10)},
  MRNUMBER = {1091506},
MRREVIEWER = {Peter\ J.\ Nicholls},
       DOI = {10.1307/mmj/1029004258},
       URL = {https://doi.org/10.1307/mmj/1029004258},
}

@article {Sam24ergodic,
    AUTHOR = {Sambarino, Andr\'es},
     TITLE = {A report on an ergodic dichotomy},
   JOURNAL = {Ergodic Theory Dynam. Systems},
  FJOURNAL = {Ergodic Theory and Dynamical Systems},
    VOLUME = {44},
      YEAR = {2024},
    NUMBER = {1},
     PAGES = {236--289},
      ISSN = {0143-3857,1469-4417},
   MRCLASS = {22E40 (37Axx 37Dxx)},
  MRNUMBER = {4676211},
       DOI = {10.1017/etds.2023.13},
       URL = {https://doi.org/10.1017/etds.2023.13},
}

@article{canary2024measurestransverse,
title = {Patterson–{S}ullivan measures for transverse subgroups},
journal = {Journal of Modern Dynamics},
volume = {20},
number = {0},
pages = {319-377},
year = {2024},
issn = {1930-5311},
doi = {10.3934/jmd.2024009},
url = {https://www.aimsciences.org/article/id/66696d275a42b314c5bf08e2},
author = {Richard Canary and Tengren Zhang and Andrew Zimmer}
}

@article{ELO23anosovmixing,
    AUTHOR = {Edwards, Samuel and Lee, Minju and Oh, Hee},
     TITLE = {Anosov groups: local mixing, counting and equidistribution},
   JOURNAL = {Geom. Topol.},
  FJOURNAL = {Geometry \& Topology},
    VOLUME = {27},
      YEAR = {2023},
    NUMBER = {2},
     PAGES = {513--573},
      ISSN = {1465-3060,1364-0380},
   MRCLASS = {22E40 (37A17 37A25 37A44)},
  MRNUMBER = {4589560},
MRREVIEWER = {Cheng\ Zheng},
       DOI = {10.2140/gt.2023.27.513},
       URL = {https://doi.org/10.2140/gt.2023.27.513},
}

@misc{Rie24code,
  title        = {{KLP} algorithm},
  author       = {Riestenberg, J. Maxwell},
  year         = 2024,
  note         = {\url{https://github.com/MaxRiestenberg/KLP-algorithm} [Accessed: (August 2024)]}
}

@misc{Weis24code,
  title        = {Anosov Check},
  author       = {Weisman, Theodore},
  year         = 2024,
  note         = {\url{https://github.com/tjweisman/anosov_check} [Accessed: (August 2024)]}
}

@article {riestenberg2024quantified,
    AUTHOR = {Riestenberg, J. Maxwell},
     TITLE = {A quantified local-to-global principle for {M}orse
              quasigeodesics},
   JOURNAL = {Groups Geom. Dyn.},
  FJOURNAL = {Groups, Geometry, and Dynamics},
    VOLUME = {19},
      YEAR = {2025},
    NUMBER = {1},
     PAGES = {37--107},
      ISSN = {1661-7207,1661-7215},
   MRCLASS = {53C35 (20F65 22E40)},
  MRNUMBER = {4862327},
       DOI = {10.4171/ggd/829},
       URL = {https://doi.org/10.4171/ggd/829},
}

@article{delarue2025locally,
  title={{Locally homogeneous {{A}xiom} {A} flows {{I}}: projective {{A}nosov} subgroups and exponential mixing}},
  author={Delarue, Benjamin and Monclair, Daniel and Sanders, Andrew},
  journal={Geometric and Functional Analysis},
  pages={1--63},
  year={2025},
  publisher={Springer}
}

@article{dancigergueritaudkassel24convexcocompactrealprojective,
  TITLE = {{Convex cocompact actions in real projective geometry}},
  AUTHOR = {Danciger, Jeffrey and Gu{\'e}ritaud, Fran{\c c}ois and Kassel, Fanny},
  URL = {https://hal.science/hal-04285877},
  NOTE = {85 pages, 9 figures},
  JOURNAL = {{Annales Scientifiques de l'{\'E}cole Normale Sup{\'e}rieure}},
  PUBLISHER = {{Soci{\'e}t{\'e} math{\'e}matique de France}},
  VOLUME = {57},
  PAGES = {1751-1841},
  YEAR = {2024},
  DOI = {10.48550/arXiv.1704.08711},
  HAL_ID = {hal-04285877},
  HAL_VERSION = {v1},
}

\end{document}